\documentclass[12pt]{amsart}

\usepackage{amsmath} \usepackage{amssymb} \usepackage{amsthm}
\usepackage{amscd}

  \def\m{{\mathfrak{m}}}\def\n{{\mathfrak{n}}}

\def\rank{{\mathrm{rank\;}}}

\DeclareMathOperator{\Soc}{Soc}
\def\q{^{[q]}}

\theoremstyle{plain}
\newtheorem{theorem}{Theorem}[section]
\newtheorem{corollary}[theorem]{Corollary}
\newtheorem{proposition}[theorem]{Proposition}
\newtheorem{conj}[theorem]{Conjecture}

\newtheorem{definition}[theorem]{Definition}
\newtheorem{discussion}[theorem]{Discussion}

\newtheorem{question}[theorem]{Question}

\newtheorem{remark}[theorem]{Remark}

\newtheorem*{acknowledgement}{Acknowledgment}

\newcommand{\bs}{\boldsymbol}

\begin{document}

\title{Bounds on the Hilbert-Kunz Multiplicity}
\date{April 18, 2011}

\author[Celikbas]{Olgur Celikbas}
\address{Department of Mathematics, University of Kansas,
Lawrence, KS 66045-7523, USA} \email{ocelikbas@math.ku.edu}

\author[Dao]{Hailong Dao}
\address{Department of Mathematics, University of Kansas,
Lawrence, KS 66045-7523, USA} \email{hdao@math.ku.edu}

\author[Huneke]{Craig Huneke}
\address{Department of Mathematics, University of Kansas,
Lawrence, KS 66045-7523, USA} \email{huneke@math.ku.edu}

\author[Zhang]{Yi Zhang}
\address{Department of Mathematics, University of Minnesota,
Minneapolis, MN 55455}
\email{zhang397@umn.edu}


\dedicatory{Dedicated to Paul Monsky}

\baselineskip 16pt \footskip = 32pt

\begin{abstract}In this paper we give new lower bounds on the Hilbert-Kunz multiplicity of
unmixed non-regular local rings, bounding them uniformly away from one. Our results
improve previous work of Aberbach and Enescu.
\end{abstract}

\maketitle \markboth{O.~Celikbas, H. Dao,  C.~Huneke, and Y. Zhang}
{Bounds on the Hilbert-Kunz multiplicity}

\section*{Introduction}

Let $(R,\m)$ be a $d$-dimensional Noetherian local ring of positive characteristic $p$. For
every $\m$-primary ideal $I$ in $R$, 
one can define an asymptotic invariant that measures the  singularity of
$R$ via the Frobenius powers of $I$. For $q$ a power of $p$, we let $I\q$ be the ideal generated by
the $q$th powers of all elements of $I$.  In 1983, Monsky \cite{M} proved that there is a
real number $e_{HK}(I) \geq 1$ such that for $q = p^e$,  
$$\ell(R/I\q) = e_{HK}(I)q^d + O(q^{d-1}).$$ 
When $I = \m$, we set $e_{HK}(\m) = e_{HK}(R)$. We use $\ell(\quad)$ to denote the length of a module,
and $\mu(\quad)$ to denote the minimal number of generators of a module throughout this paper. 

Later Huneke, McDermott and Monsky \cite{HMM} extended this result in the case $R$ is integrally closed and
excellent with perfect residue field
to show there is a constant $\beta$ such that
$$\ell(R/I\q) = e_{HK}(I)q^d + \beta q^{d-1}+ O(q^{d-2}).$$
See \cite{HoY} for further refinements.

The number $e_{HK}(R)$ is called the Hilbert-Kunz multiplicity of $R$. 
If $R$ is regular it is easy to see that $e_{HK}(R) = 1$.
The converse is true if $R$ is unmixed, but this is highly non-obvious. In fact, it was not until
2000 that Watanabe and Yoshida \cite{WY1} proved this fact. A somewhat simpler proof can be found in \cite{HY}. 

In this paper we are interested in giving lower bounds on the Hilbert-Kunz multiplicity. 
Let $e$ denote the usual Hilbert-Samuel multiplicity of
$R$. It is not hard to prove the following bounds:

$$\text{max}\{\frac {e}{d!},\, 1\}\leq e_{HK}(R)\leq e.$$
The last inequality, for example, follows at once from Lech's Lemma \cite[11.2.10]{SH}.

A natural way to think about the Hilbert-Kunz multiplicity is that as it nears $1$,
the singularities of the ring should be better.
Blickle and Enescu were able to prove an explicit version of this principle:

\medskip
\begin{theorem}[Blickle-Enescu, 2004 \cite{BE}]\label{b-e} Let $(R,\m,k)$ be an unmixed local ring of
prime characteristic $p > 0$.
\newline 1) If $e_{HK}(R) < 1 + \frac {1}{d!}$, then $R$ is Cohen-Macaulay and F-rational.
\newline 2) If $e_{HK}(R) < 1 + \frac {1}{p^dd!}$, then $R$ is regular.
\end{theorem}

A Noetherian local ring $R$ of prime characteristic $p$ is said to be \it F-rational \rm if ideals generated by
systems of parameters are tightly closed. For details concerning the theory of tight
closure and its relationship to the Hilbert-Kunz multiplicity, we refer to \cite{HH} and \cite{H}.
We will not be directly using tight closure in this paper, except at the very end of the paper.

The second result in Theorem \ref{b-e} is unsatisfactory due to its dependence on $p$.
Blickle and Enescu \cite{BE} raised the question of whether
or not the Hilbert-Kunz multiplicity is uniformly bounded away from $1$ for unmixed local rings
which are not regular. Specifically, the following was asked:

\smallskip

\begin{question}\label{question} Is there a constant $\epsilon(d)$, depending only on the dimension $d$ of $R$, such that if
$R$ is an unmixed local ring which is not regular, then $e_{HK}(R) -1 \geq \epsilon(d)$?
\end{question}
\smallskip

As a further question, Watanabe and Yoshida asked that if one can determine the infinmum of
Hilbert-Kunz multiplicities over all unmixed local Noetherian rings of characteristic $p$ and
dimension $d$ which are not regular, and if that infinmum is attained, to classify those rings which have the lowest value.

Aberbach and Enescu gave a positive answer to Question \ref{question} in \cite{AE}. Our main result, Theorem~\ref{mainthm},
improves upon their bound.

One could extend these questions to ask whether or not the Hilbert-Kunz multiplicities of 
rings of fixed characteristic and dimension are discrete within any bounded region, or are
there limit points?

Question \ref{question} was made even more explicit in \cite{WY3}, where the following conjecture is posed:

\begin{conj} Let $d \geq 1$ be an integer and $p > 2$ be a prime number. Set
$A_{p,d} :=K[[X_0,X_1,...,X_d]]/(X_0^2 +\ldots +X_d^2)$, where $K$ is the algebraic closure of the field with $p$
elements.  Let $(R, \m, K)$ be a $d$-dimensional
unmixed local ring. Then the
following two statements hold.
(1) If $R$ is not regular, then $e_{HK}(R) \geq  e_{HK}(A_{p,d}) \geq  1+ \frac{c_d}{d!}$. 
(2) If $e_{HK}(R) = e_{HK}(A_{p,d})$, then $\widehat{R}\cong A_{p,d}$.
\end{conj}

In this conjecture, $c_d$ is defined by the equation, 
$$\mathrm{sec}(x) + \mathrm{tan}(x) = \sum_{d=0}^{\infty} \frac {c_d}{d!}x^d.$$
Monsky proved that $\mathrm{lim}_{p\rightarrow \infty}  e_{HK}(A_{p,d}) = 1 + \frac{c_d}{d!}$.
To get a feeling for what the conjecture says for small dimension, the coefficients $\frac{c_d}{d!}$
are as follows: $\frac{c_2}{2!} = \frac{1}{2}$, $\frac{c_3}{3!} = \frac{1}{3}$,
$\frac{c_4}{4!} = \frac{5}{24}$, and $\frac{c_5}{5!} = \frac{2}{15}$.
Enescu and Shimomoto prove the conjecture for complete intersections in \cite{ES}.

In \cite{WY3} and \cite{WY4} both the problem of lower bounds and classification are solved in dimensions at most four. Moreover
in low dimension they showed that there is a minimal value for the Hilbert-Kunz multiplicity for non-regular rings,
and rings which take this value are classified. 
For example, they proved
that if  $(R,\m,k)$ is a two-dimensional Cohen-Macaulay local ring of prime positive characteristic $p$, algebraically closed
residue field,  and multiplicity $e$
then $e_{HK}(R)\geq \frac {e+1}{2}$ with equality if and only if 
the associated graded ring of $R$ is isomorphic
to the $e$th Veronese subring of $k[x,y]$.
They also proved that 
if  $(R,\m,k)$ is a three-dimensional unmixed local ring of positive prime characteristic $p$ which is not regular,
then $e_{HK}(R) \geq \frac {4}{3}.$ Furthermore, 
if $k$ is algebraically closed of characteristic not $2$, then equality occurs
if and only if  $R$ is analytically isomorphic with $k[[x,y,z,w]]/(x^2+y^2+z^2+w^2)$.

Using an intricate and beautiful argument, Aberbach and Enescu solved the basic problem of bounds independent of the
characteristic in all dimensions:

\begin{theorem}[Aberbach-Enescu]\label{a-ebound} Let $(R,\m,k)$ be an unmixed ring of characteristic
$p$ and dimension $d\geq 2$.
If $R$ is not regular, then
$$e_{HK}(R)\geq 1 + \frac{1}{d(d!(d-1)+1)^d}.$$
\end{theorem}

Our main results in this paper build from the work of Aberbach and Enescu.\footnote{As this paper was being completed,
Aberbach and Enescu posted on the ArXiv a new paper with better bounds that their previous work. We will compare these bounds to ours in the last
section; our bounds given in Theorem \ref{mainthmintro} are better in general. See Remark \ref{remarka-e}.} 
In the next section  we give a bound in the Gorenstein case which
is often better than bounds in \cite{AE}. Our improvement comes  by introducing the
F-signature into the proofs, and consideration of a certain dual sequence.  However, our main theorem, in Section 3, 
is an improvement of Theorem~\ref{a-ebound}:

\begin{theorem}\label{mainthmintro} Let $(R,\m,k)$ be a Noetherian local unmixed ring of prime characteristic $p$ with infinite perfect residue field and dimension $d\geq 2$.
Let $\bs{x}$ be a minimal reduction of $\m$, and let $\mu$ be the maximal number of minimal generators of any ideal in
$R/(\bs{x})$. Let $t$ be the largest integer such that $\overline{\m^t}$ is not contained in $(\bs{x})$.
If $R$ is not regular, then
$$e_{HK}(R)\geq 1 +  (\text{min}\{\frac{1}{d!}, (\frac{\mu}{e-\mu}) \cdot \frac{1}{(\lceil\frac{d}{t}\rceil)^d}\}).$$
\end{theorem}

Here and throughout the paper, by $\overline{J}$ we denote the integral closure of an ideal $J$.
Though the bound in Theorem \ref{mainthmintro}  is a substantial improvement on Theorem \ref{a-ebound}, it needs to be pointed out that our bound is still probably
far from best, as suggested by the conjecture of Watanabe and Yoshida above.
However, the method we use has some interest in its own right, and may be useful in other contexts.
For instance in Proposition \ref{Wat}, we use these ideas to give an affirmative answer to an old conjecture
of Watanabe in many cases.

\vskip.3truein

\section{New Estimates}

\bigskip

In this section we give new bounds  on the Hilbert-Kunz multiplicity which depend upon the F-signature.
We first recall that a Noetherian ring $R$ of prime characteristic $p$ is said to be \it F-finite \rm if
$R$ is a finitely generated $R$-module via the Frobenius.

\begin{definition} Let $(R,\m,k)$ be a Noetherian local ring of prime charactersitic $p$ and dimension $d$. Assume that
$R$ is F-finite, and reduced. For $q = p^e$, write $R^{1/q} \cong R^{a_q}\oplus M$, where $M$ has no free summands.
Set $\alpha(R) = \text{log}_p [k : k^p]$. Consider the sequence $\{\frac{a_q}{q^{d+\alpha(R)}}\}$. 
We denote by $s_{-}(R)$ and $s^+(R)$ the liminf and
limsup respectively of the above sequence as $q \rightarrow \infty$. If $s_{-}(R) = s^+(R)$, then
the limit, denoted by $s(R)$, is called the F-signature of R.
\end{definition}

Throughout the rest of the paper we will  write $s$ for $s(R)$ if there is no ambiguity about the ring $R$.

In very recent and striking work, Kevin Tucker \cite{T} has shown that the F-signature exists in general. In our main technical results the ring will be Gorenstein, in such case the existence of $F$-signature is known and easy to prove. 
For the statement of our first result, Theorem \ref{T:ee}, it is worth noting that the F-signature is always at
most $1$, and is equal to $1$ if and only if the ring is regular \cite[Corollary 16]{HL}.
Thus the fraction appearing in this theorem is always at least $1$, and is only exactly $1$
when the ring is regular.

\begin{theorem} \label{T:ee}
Suppose that  $(R,\m,k)$ is a reduced F-finite Gorenstein local ring of prime characteristic $p$ with
infinite perfect residue field.  Let
$e$ denote the multiplicity of $R$, and let $s$ denote the F-signature of $R$. We choose a minimal reduction $\bs{x}$ of the maximal ideal. 
For any ideal $I$ of $R$ such that $I\supseteq (\bs{x}),$ let
$\mu=\mu(I/(\bs{x})),$ then $$e_{HK}(R)\geq \frac{e-s\mu}{e-\mu} = 1 + \frac{\mu}{e-\mu}(1-s).$$
\end{theorem}

\begin{proof}
We write $e_{HK}$ for $e_{HK}(R)$. Let $F=R^{\mu(R^{1/q})}$. In the proof of this theorem
we use the notation $M^* = \text{Hom}_R(M,R)$. Note that $(R^{1/q})^*\cong R^{1/q}$ since $R$ is Gorenstein: the dual
of $R^{1/q}$ is isomorphic to the canonical module of $R^{1/q}$, and since $R^{1/q}$ is also
Gorenstein, this canonical module is isomorphic to the ring. Hence there is a short exact sequence
$0\to N\to F\to (R^{1/q})^*\to0.$
Its dual is exact since $(R^{1/q})^*$ is maximal Cohen-Macaulay and $R$ is Gorenstein:
$$0\to R^{1/q}\to F^*\to N^*\to 0$$
(see  \cite[Theorem 3.3.10(c)]{BH}).
We now write $R^{1/q}\cong R^{a_q}\oplus M$, where $M$ has no free summand. Note that 
this means $a_q$ copies of $R$ split from the above short exact sequence, giving rise to
an exact sequence,
$$0\to M\to G^*\to N^*\to 0$$
where $G = R^{\mu(R^{1/q})-a_q}$. 
Observe that the image of $M$ is contained in $\m G^*$ since $M$ has no free summand
(if not, then dualizing would give a free summand of $M^*$, and hence of $M^{**}\cong M$).

Since $\mu(F)=\ell(R/\m^{[q]}) = e_{HK}q^d+O(q^{d-1}),$ 
we can write $G^* = R^{e_{HK}q^d-a_q+d_q}$, where $d_q = O(q^{d-1})$.  Since $N$ is a
maximal Cohen-Macaulay module (MCM), so is $N^*$ by \cite[Theorem 3.3.10(d)]{BH}. 
Then $\text{Tor}_1^R(N^*,R/(\bs{x}))=0.$ Therefore, we have an injection 
$$\frac{M}{\bs{x}M} \hookrightarrow \frac{G^*}{\bs{x}G^*},$$
Then $$\frac{IR^{1/q}}{\bs{x}R^{1/q}} = \frac{IR^{a_q}}{\bs{x}R^{a_q}}\oplus \frac{IM}{\bs{x}M}
\hookrightarrow \frac{IR^{a_q}}{\bs{x}R^{a_q}} \oplus \frac{(\m I+\bs{x})G}{\bs{x}G}.$$ 
Computing the length, we see that
\begin{align*}
&\ell(IR^{1/q}/\bs{x}R^{1/q})=\ell(I^{[q]}R/\bs{x}^{[q]}R)\\ =&\ell(R/\bs{x}^{[q]}R)- \ell(R/I^{[q]})\\
\leqslant& a_q\ell(I/(\bs{x})) +\ell((\m I+(\bs{x}))/(\bs{x}))(e_{HK}q^d-a_q+d_q).
\end{align*}
Since $R$ is Cohen-Macaulay, we obtain that
$$\ell(R/\bs{x}^{[q]}R) = eq^d\leq \ell(R/I^{[q]})+ a_q\ell(I/(\bs{x})) +\ell((\m I+(\bs{x}))/(\bs{x}))(e_{HK}q^d-a_q+d_q).$$

Let $I\subset J_1\subset J_2\subset \cdots \subset R$ be a cyclic filtration of ideals of $R$ such that $J_{i+1}=J_i+(s_i)$ and 
such that $J_{i+1}/J_i\cong k$.
Since there are surjections $R/\m^{[q]}\twoheadrightarrow J^{[q]}_{i+1}/J^{[q]}_i,$ we see that
$\ell(R/I^{[q]})\leqslant \ell(R/\m^{[q]})\cdot \ell(R/I).$ 
Therefore, taking the $q$th Frobenius of the filtration of $I\subset R,$ we get
\begin{align*}
eq^d&\leqslant \ell(R/I^{[q]})+a_q\ell(I/(\bs{x}))+\ell(\frac{\m I+(\bs{x})}{(\bs{x})})(e_{HK}q^d-a_q+d_q)\\
&\leqslant \ell(R/\m^{[q]})\ell(R/I)+a_q\mu+e_{HK}q^d\ell(\frac{mI+(\bs{x})}{(\bs{x})})+d_q\ell(\frac{mI+(\bs{x})}{(\bs{x})}).\\
\end{align*}
(Note that $a_q\mu = a_q\ell(I/(\m I + (\bs{x}))$.)
Dividing by $q^d$ and taking limits as $q\rightarrow \infty$, and observing that $\text{lim}\frac{a_q}{q^d} = s$ and
$\ell(R/(\bs{x})) = e$,  we see that
$e\leqslant e_{HK}(e-\mu)+s\mu,$ hence \[\frac{e-s\mu}{e-\mu}\leqslant e_{HK}.\qedhere\]
\end{proof}

This result can be used to recover and sometimes improve estimates in \cite{AE}. For example,
compare the following to \cite[Cor. 3.7]{AE}.

\begin{corollary}\label{bound} Let  $(R,\m,k)$ be a reduced F-finite Gorenstein local ring of prime characteristic $p$ with
infinite perfect residue field.
Then $$e_{HK}(R)\geq \frac{e-s(v-d)}{e-v+d},$$ where $v$ is the embedding dimension of $R$, i.e., the number of
minimal generators of $\m$.  If $R$ is not F-regular, then $$e_{HK}(R)\geq \frac{e}{e-v+d},$$
\end{corollary}

\begin{proof}  The proof of the first statement follows immediately from Theorem~\ref{T:ee} by applying this theorem to
the ideal $I = \m$. The second statement holds since if $R$ is not F-regular, then $s= 0$ \cite{HL}
(the converse is also true, see \cite{AL}).
\end{proof}

\begin{remark}\label{F-sing}{\rm If $R$ is not F-rational  (or even not strongly F-regular), then $s(R) = 0$, and Corollary \ref{bound}, already in \cite{AE}
in this case, gives a very good uniform bound away from $1$ for the Hilbert-Kunz multiplicity. In fact, when $R$ is Gorenstein, $s(R)>0$ if and only if $R$ is $F$-rational (equivalently in this case, strongly $F$-regular, weakly $F$-regular or $F$-regular).

In general, even in the F-rational case,
we would have a very good estimate bounding the Hilbert-Kunz multiplicity away from $1$, if we could bound $s$ below $1$
uniformly. In fact the problem of bounding the F-signature below $1$ and that of bounding the Hilbert-Kunz multiplicity above $1$ are
in some sense the same problem. In \cite[Proposition 14]{HL}, the following upper bound is given:
$$(e-1)(1-s) \geq e_{HK}(R) -1.$$ Combining this with Theorem \ref{T:ee} and rewriting gives:
$$\frac {\mu}{e-\mu}\leq \frac {e_{HK}(R)-1}{1-s}\leq e-1.$$
This means that the ratio in the middle of the two terms we are interested in are trapped between numbers depending on the usual
multiplicity of the ring and the number $\mu$. Note that equality holds between the top and bottom terms in the inequality if
and only if $\mu = e-1$ if and only if $R$ has minimal multiplicity.}
\end{remark}

\vskip.3truein

\section{Chains of Integrally Closed Ideals}
\bigskip
We begin this section with a seemingly unrelated result. The idea behind this is found in the
thesis of Choi \cite{Ch}. Note that we do not assume $R$ to be of prime characteristic.

\begin{proposition} \label{ranklemma} Let $(R,\m)$ be a local Noetherian domain, and
let $I=(J,u)$ where $J$  is an integrally closed $\mathfrak{m}$-primary ideal of $R$ and $u\in J:\m$.
If $M$ is a finitely generated torsion-free $R$-module, then $$\ell(IM/JM)\geq \rank M. $$
\end{proposition}

\begin{proof}
Set $N=(JM:_Mu)$. Since $\displaystyle{\frac{M}{N}\cong \frac{(J,u)M}{JM}}$ and $\mathfrak{m}M\subseteq N$, we can write $M=N+N'$ with $\displaystyle{\mu(N')=\ell\left(\frac{IM}{JM}\right)}$.
Thus it suffices to prove $\mu(N') \geq \rank(M)$. Since $u(M/N')\subseteq J(M/N')$, it follows from the
determinantal trick \cite[2.1.8]{SH} that there is an element $ r = u^n+j_{1} \cdot u^{n-1}+\cdots + j_{n}$
with $j_{i}\in J^{i}$ for all $i$ such that 
$rM \subseteq N'$.
Observe that $r \neq 0$ since $J$ is integrally closed and $u\notin J$. Since
$M_r = N'_r$, this implies that   $\mu(N')\geq \rank(N')=\rank(M)$.
\end{proof}

Given two ideals $I$ and $J$ with $J \subseteq I$, $\overline{\ell}(I/J)$ will denote the length of the longest chain of integrally closed ideals between $J$ and $I$.

\begin{corollary} \label{rankcor} Let $(R,\m)$ be a Noetherian local domain.  Let $J$ be an integrally closed $\mathfrak{m}$-primary ideal of $R$ and let $I$ be an ideal containing $J$.
If $M$ is a finitely generated torsion-free $R$-module, then $$\ell(IM/JM)\geq \overline{\ell}(I/J) \cdot \rank(M).$$
\end{corollary}

\begin{proof} Set $n = \overline{\ell}(I/J)$. Then  there is a chain of ideals $$J=K_0\subset K_1\subset \ldots \subset K_{n-1}\subset K_{n}\subset I$$ with $\overline{K_{i}}=K_{i}$ for all $i$. Then
$$\ell\left(IM/JM\right) \geq  \sum_{j=0}^{n} \ell(K_{j+1}M/K_{j}M) \geq \sum_{j=0}^{n} \ell((K_{j},u_{j})M/K_{j}M)$$
for some $u_{j}\in K_{j+1} \cap \Soc(K_{j})$. Thus the result follows from Proposition \ref{ranklemma}.
\end{proof}

In \cite[2.17]{WY1}, the following conjectures were raised:

\smallskip

{\it  Let $A$ be a Cohen-Macaulay local ring of characteristic
$p > 0$. Then for any $\m$-primary ideal $I$, we have
   (1) $e_{HK}(I) \geq \ell(A/I).$
   (2) If pd$_A(A/I) < \infty$, then $e_{HK}(I) = \ell(A/I)$.}

\smallskip

Although this conjecture has turned out to not be true (for example, see the paper of Kurano \cite{Ku}), our next result shows that
(1) is true for many $\m$-primary ideals:

\begin{proposition}\label{Wat} Assume $(R,\m)$ is an excellent normal ring of prime characteristic $p$ with an algebraically closed residue field. If $I$ is an integrally closed $\mathfrak{m}$-primary ideal of $R$, then
$$e_{HK}(I) \geq \ell(R/I)+e_{HK}(R)-1.$$
If $I$ is an $\m$-primary ideal such that there is an integrally closed ideal $K\subset I$ with
$\ell(I/K) = 1$, then 
$$e_{HK}(I) \geq \ell(R/I).$$
\end{proposition}

\begin{proof} We prove the first statement. From \cite[2.1]{W}, we obtain that there is a composition series $\displaystyle{I=I_0\subset I_1\subset \cdots \subset I_{l-1}=\mathfrak{m}\subset R}$
such that $\overline{I}_i=I_i$ for all $i$.
It follows from Corollary \ref{rankcor} that:
$$\ell \left(\frac{\m R^{1/q}}{IR^{1/q}}\right)\geq \ell(\m/I) \cdot \rank(R^{1/q}). $$
Therefore,
$$\ell \left(\frac{R^{1/q}}{IR^{1/q}}\right)\geq \ell(\m/I) \cdot \rank(R^{1/q})+ \ell \left(\frac{R^{1/q}}{\m R^{1/q}}\right).$$
Dividing by $q^d,$ we get \[e_{HK}(I)\geq \ell(R/I)+e_{HK}(R)-1.\]

We prove the second statement. Note that $\ell(R/K^{[q]}) = \ell(R/I\q) + \ell(I\q/K^{[q]}).$
On the other hand, $\ell(I\q/K^{[q]})\leq \ell(R/\m^{[q]})$. Hence $\ell(R/I\q) + \ell(R/\m^{[q]})\geq \ell(R/K^{[q]})$.
Dividing by $q^d$ and taking limits, we obtain that
$$e_{HK}(I)+e_{HK}(R)\geq e_{HK}(K) \geq \ell(R/K) + e_{HK}(R) - 1$$ by the first
part of the theorem. The result follows.
\end{proof}

In \cite{WY1} (see Theorem 2.7), Watanabe and Yoshida raised the following question:
If  $(R,\m)$ is an unmixed local ring of prime characteristic $p$, then is it true that for every $\m$-primary integrally closed
ideal $I$, that $e_{HK}(I)\geq \ell(R/I)$? Furthermore, if equality holds for some $I$, does it force $R$ to be regular.

Obviously, the first part of this question is answered in the positive by Proposition~\ref{Wat} above provided $R$ is in addition
excellent normal with algebraically closed residue field. We are grateful to the referee for providing the following extension, which
also answers positively the second question:

\begin{corollary}\label{ref} Let $(R,\m)$ be an excellent local analytically irreducible domain of prime characteristic $p$ with an algebraically closed residue field. If $I$ is an integrally closed
$\m$-primary ideal then $e_{HK}(I)\geq \ell(R/I)$, with equality if and only if $R$ is regular.
\end{corollary}

\begin{proof} (due to the anonymous referee) If $R$ is normal, the inequality is given in Proposition~\ref{Wat}. The normalization
$S$ of $R$ is an excellent normal local domain, and it is a finitely generated $R$-module. Let $\n$ be the maximal ideal of $S$. Then
$S/\n = R/\m$. In particular, the length of an $S$-module is the same as the length of the same module viewed as an $R$-module, so
in computing length we don't need to specify which ring we are working over. Set $L = \overline{IS}$. Then $L$ is an integrally closed ideal of $S$ and $L\cap R = I$. Thus we have that
$e_{HK}(I)  = e_{HK}(IS) \geq e_{HK}(L) \geq \ell(S/L)$ by Proposition~\ref{Wat}, and $\ell(S/L)\geq \ell(R/I)$, since $R/I\subset S/L$. 

For the last claim, suppose first that $R$ is normal but not regular. Then $e_{HK}(I) > \ell(R/I)$ by Proposition~\ref{Wat}. If $R$ is not normal
and we have equality, then equality must occur in all the inequalities in the paragraph above. Then $S$ is regular, and $e_{HK}(IS) = e_{HK}(L)$.
By basic facts of tight closure theory, see \cite{HH}, it follows that $IS = L$. Moreover, from the same set of inequalities, we must have that $R/I = S/L$. Thus $S = R +IS$ as an
$R$-module and Nakayama's lemma gives that $R = S$ is regular.
\end{proof}
\vskip.3truein

\section{Main Result}

\bigskip

Aberbach and Enescu improved the Blickle-Enescu Theorem~\ref{b-e} (2) by proving:

\begin{theorem}[Aberbach-Enescu, 2008, \cite{AE}]\label{a-e2} Let $(R,\m,k)$ be an unmixed ring of prime characteristic $p$ and dimension $d\geq 2$. If
$e_{HK}(R)\leq 1 + \text{max}\{\frac {1}{d!},\, \frac {1}{e}\}$, then $R$ is Gorenstein
and F-regular.
\end{theorem}
\medskip

In particular, for the purposes of bounding the Hilbert-Kunz multiplicity away from $1$, we may assume that $R$ is Gorenstein and F-regular.
Recall that F-regular means that every ideal in every localization of $R$ is tightly closed (see \cite{H} for the definition
of tight closure and basic properties. We will use that F-regular rings are Cohen-Macaulay and normal.) 

We will use the following fact, which appears in \cite[Theorem 2.7]{WY1}, for example:

\begin{theorem}\label{finiteext} Let $(R,\m)\rightarrow (S,\n)$ be a module-finite extension of Noetherian local domains.
Then for every $\m$-primary ideal $I$ of $R$,
$$e_{HK}(I)\cdot [Q(S) : Q(R)] =e_{HK}(IS)\cdot [S/\n : R/\m]$$
where $Q(\quad)$ denotes the field of fractions of a domain.
\end{theorem}

We also need the following discussion, which is found in \cite[Remark 4.3]{AE}.

\begin{discussion}\label{disc}
{\rm Let $(R,\m)$ be a local domain. Let $z \in \m$, and let $n$ be a positive integer. 
Let $y\in R^+$ be any root of $f(X)=X^n-z$. We call $S=R[y]$ a radical extension for the pair $R$ and $z$.

Whenever $S$ is radical for $R$ and $z$, then $b := [Q(S) : Q(R)] \leq n$. Assume
also that $R$ is normal and $z$ is a minimal generator of $\m$. Then in fact, $b = n$. Moreover,
$S=R[y]\cong R[X]/(X^n-z)$ so that if $R$ is Cohen-Macaulay (respectively Gorenstein) then so is $S$. Also, $S$ is local with maximal ideal $(\mathfrak{m}, y)$, 
because $\m$ is certainly in the Jacobson radical of $S$ as $S$ is a finite extension of $R$, and
$S/\m S\cong k[X]/(X^n)$ where $k$ is the residue field of $R$.}
\end{discussion}

\begin{proposition} \label{intprop} Assume $(R,\m)$ is Cohen-Macaulay and normal, and let $x\in \mathfrak{m}-\mathfrak{m}^2$ be part of a minimal reduction 
of $\mathfrak{m}$. Let $S=R[y]$ with $y^n=x$. Then $\m S+(y^i)$ is integrally closed for any nonnegative integer $i$.
\end{proposition}
\begin{proof} Assume the claim is wrong. We may assume that $i\leq n$. Since $S/\m S\cong k[X]/(X^n)$, the ideals in $S/\m S$ are linearly ordered, and thus  $y^{i-1}\in \overline{\mathfrak{m}S+(y^i)}$. By the valuative criterion \cite[6.8.3]{SH},  $\overline{\mathfrak{m}S+(y^i)}=\bigcap (\mathfrak{m}S+(y^i))V\cap R$ where $(V,v)$ runs over all discrete valuation domains of the field of fractions of $R$
centered on the maximal ideal of $R$.
Therefore $$v(y^{i-1})\geq v(\mathfrak{m}S+(y^i))\geq \min \{v(\mathfrak{m}S),v(y^i)\},$$ which is equivalent to $$(i-1)v(y)\geq \min\{v(\mathfrak{m}S),iv(y)\}.$$
Hence $(i-1)v(y)\geq v(\mathfrak{m}S),$ that is, $\displaystyle{y^{i-1}\in \overline{\mathfrak{m}S}}$. By \cite[6.8.12]{SH},
this is equivalent to $$c\cdot (y^{i-1})^{nl}=c\cdot (x^{i-1})^l \in \mathfrak{m}^{nl}S,$$ for some $c\in R$ and for all $l\gg 0$.
Since $R$ is normal, Discussion \ref{disc}  shows that $S$ is a free $R$-module. Therefore
$$c\cdot (x^{i-1})^l \in \mathfrak{m}^{nl}.$$
Thus $\displaystyle{x^{i-1} \in \overline{\mathfrak{m}^{n}}}$.
Now set $j=i-1$ and choose a minimal reduction $(x,x_2,\ldots,x_d)$ of $\mathfrak{m}$.
Then $\displaystyle{x^j\in \overline{(x,x_2,\ldots, x_d)^n}}$ and hence $\displaystyle{x^{jl}\in (x,x_2,\ldots,x_d)^{n(l-t)}}$
for some $t$ and for all $l\gg 0$. Since $x,x_2,\ldots,x_d$ is a system of parameters, this is not possible.
\end{proof}

\begin{corollary}\label{RS}
Assume that $(R,\m,k)$ is a Cohen-Macaulay normal local ring of prime characteristic $p$ which is F-finite with
infinite perfect residue field. Let $x\in \mathfrak{m}-\mathfrak{m}^2$ 
be part of a minimal reduction of $\mathfrak{m}$ and let $S=R[y]$ with $y^n=x.$ 
Then $$e_{HK}(R)-1\geq \frac{e_{HK}(S)-1}{n}.$$
\end{corollary}

\begin{proof}
We use Proposition \ref{intprop} and Corollary \ref{rankcor}. By  Proposition \ref{intprop}, there is a chain of integrally closed
ideals, $\m S\subset (\m S + (y^{n-1}))\subset \ldots \subset (\m S + (y)),$ and so using the torsion-free $S$-module $S^{1/q}$ and applying  Corollary \ref{rankcor} to the pair of integrally closed ideals $\m S\subset (\m S, y)$, 
yields that $\ell(S^{1/q}/\m S^{1/q})\geq (n-1)\text{rank}(S^{1/q}) + \ell(S^{1/q}/(\m S, y)S^{1/q})$. Dividing by $q^d$ and taking limits gives  that $$e_{HK}(\m S)\geq n +e_{HK}(S)-1$$
Moreover, $e_{HK}(\mathfrak{m}S)=
n\cdot e_{HK}(R)$ by Theorem \ref{finiteext}.
Therefore, $$n\cdot e_{HK}(R)\geq n+e_{HK}(S)-1$$ and hence the result follows.
\end{proof}

The last corollary is in some sense our main improvement upon the methods used by Aberbach and Enescu. We use their
strategy of adjoining roots of a minimal reduction until we obtain a non F-regular ring; at this point previous
estimates are good. The new work of Aberbach and Enescu \cite{AE1} which appeared in ArXiv as we were finishing this
paper gives the following result, among other estimates in lower dimension. Their method of comparing the Hilbert-Kunz
multiplicities of radical extensions seem much different from the one we developed above.
The new result of Aberbach and Enescu \cite{AE1} is:

\begin{theorem}\label{newa-e} Let $(R,\m,k)$ be a local Gorenstein F-regular ring of dimension $d \geq 2$ and Hilbert-Samuel multiplicity $e \geq 6$.
Let $k = \mu(\m) -\dim(R)$. Assume further that R is not a complete intersection.
Then if $e\geq d!+1$, then $e_{HK}(R)\geq 1+\frac{1}{d!}$. Otherwise, if $k = e-2$, then
$$e_{HK}(R)\geq 1 + 3(\frac {4}{6(\lceil\frac{d}{2}\rceil)-2})^d$$
while if $k\ne e-2$, then
$$e_{HK}(R)\geq 1 + (\frac {4}{(\lceil\frac{d}{3}\rceil)d!+4})^d(\frac{4}{(6d-16)}).$$
\end{theorem}

Our main result is the following:

\begin{theorem}\label{mainthm}
Let $(R,\m,k)$ be a Noetherian local unmixed ring of prime characteristic $p$ which is F-finite  with infinite perfect residue field and dimension $d\geq 2$.
Let $(\bs{x})$ be a minimal reduction of $\m$, and let $\mu$ be the maximal number of minimal generators of any ideal in
$R/(\bs{x})$. Let $t$ be the largest integer such that $\overline{\m^t}$ is not contained in $(\bs{x})$.
If $R$ is not regular, then
$$e_{HK}(R)\geq 1 +  (\text{min}\{\frac{1}{d!}, (\frac{\mu}{e-\mu}) \cdot \frac{1}{(\lceil\frac{d}{t}\rceil)^d}\}).$$
\end{theorem}

\begin{proof}
If $e_{HK}(R)\geq 1+ 1/d!,$ there is nothing to prove. Hence we may assume that
$e_{HK}(R)< 1+ 1/d!,$ and  then $R$ is $F$-regular and Gorenstein by \cite[3.6]{AE}. Thus we may assume that $R$ is $F$-regular and Gorenstein.

Let $(\bs{x})=(x_1,\cdots,x_d)$ be a minimal reduction of $\m,$
let $\mu=\mu(\m/(\bs{x})).$ Consider the set of overrings $S=R[x^{1/n}_1, \dots, x_{i}^{1/n}]=R_{i,n}$ which are
not $F$-regular. Choose $n$ and $i$ such that we attain
$\text{min } \{n^i: R_{i,n} \text{ is not F-regular} \}$. Set $S = R_{i,n}$.
Then by Theorem \ref{T:ee} applied to $S$ and the minimal reduction $x_1^{1/n},...,x_i^{1/n}, x_{i+1},...,x_d$ and the fact that $s(S)=0$ (see Remark \ref{F-sing}),
$$e_{HK}(S)\geq \frac{e(S)-s(S)\cdot \mu(S)}{e(S)-\mu(S)}= \frac{e(S)}{e(S)-\mu(S)}.$$
Here, we define  $\mu(S)$ to be the maximal number of minimal generators of any ideal in
$S/(x_1^{1/n},...,x_i^{1/n}, x_{i+1},..,x_d)$.

However, since $S/(x_1^{1/n},...,x_i^{1/n}, x_{i+1},..,x_d)\cong R/(\bs{x}),$ we have $e(S)=e(R)$, and $\mu(S) = \mu$. Therefore,
$$e_{HK}(S)\geq 1 + \frac{\mu}{e-\mu}.$$

Let $R_0=R,$ and for each $i\geq j\geq 1,$ let $R_j=R_{j-1}[x^{1/n}_j],$ then by Corollary \ref{RS} $$e_{HK}(R_j)-1\geq \frac{e_{HK}(R_{j-1})-1}{n}.$$

Therefore, we get $$e_{HK}(R)-1=e_{HK}(R_0)-1\geq \frac{e_{HK}(S)-1}{n^i}\geq \frac{\mu}{e-\mu}(\frac{1}{n^i}). $$

Thus it remains to prove that 
$$\text{min } \{n^i: R_{i,n} \text{ is not F-regular} \}\leq (\lceil\frac{d}{t}\rceil)^d.$$

To do this we note that it suffices to prove that $A= R[x_1^{1/\lceil\frac{d}{t}\rceil},...,x_d^{1/\lceil\frac{d}{t}\rceil}]$
is not F-regular.  
Set $y_i=x_{i}^{1/\lceil\frac{d}{t}\rceil}, i=1,\ldots, d$. Then a socle representative  of $A/(\bs{x})$ is:
$$u \cdot y_{1}^{\lceil\frac{d}{t}\rceil-1} \ldots y_{d}^{\lceil\frac{d}{t}\rceil-1},$$
where $u$ generates the socle of $(\bs{x}R)$. By our assumption on $t$, we may in addition assume that $u\in\overline{\m^t}$.
Let $v$ be any discrete valuation centered on the maximal ideal of $S$.
Then we claim that
$$v(u) + (((\lceil\frac{d}{t}\rceil-1)d)/\lceil\frac{d}{t}\rceil)v(\m) \geq dv(\m).$$
Since $v(u)\geq tv(\m)$, it suffices to prove that
$$t + (((\lceil\frac{d}{t}\rceil-1)d)/\lceil\frac{d}{t}\rceil)\geq d.$$
This simplifies to $t(\lceil\frac{d}{t}\rceil)\geq d$, which is true.

It follows that $u \cdot y_{1}^{\lceil\frac{d}{t}\rceil-1} \ldots y_{d}^{\lceil\frac{d}{t}\rceil-1}\in
\overline{(\m S)^d}$.
By the tight closure Brian\c con-Skoda theorem \cite[Section 5]{HH} this implies that $(x_1^{1/\lceil\frac{d}{t}\rceil},...,x_d^{1/\lceil\frac{d}{t}\rceil})A$
is not tightly closed, which gives the desired conclusion. 
\end{proof}

\begin{remark}{\rm Aberbach and Enescu also state \cite[Page 15]{AE} an inequality in their
text which is closer to the one we give. Let $r$ be the maximum $i$
such that $\m^i$ is not contained in $(\bs{x})$. (Note that $r\leq t$.)
Let $i_0$ be the least integer such that $R_i = R[x_1^{1/m},...,x_i^{1/m}]$ is
not F-regular, where $m = \lceil\frac{d}{r}\rceil$. (Observe that $m\geq \lceil\frac{d}{t}\rceil$.) Then
$$e_{HK}(R) \geq 1 + (\frac{1}{e(m-1)+1)})^{i_0}(\frac{1}{d}).$$
In this situation the estimate is closely related to our estimate above which we give  with
a multiple of  $\frac{1}{n^i}$; the main difference now is that the
estimate of Aberbach and Enescu has an extra $e^{i_0}$ in the denominator.}
\end{remark}

\smallskip

\begin{remark}\label{remarka-e}{\rm To compare this theorem to the new result of Aberbach and Enescu, note that since we
may assume that $R$ is Gorenstein, $t\geq 2$ unless $R$ is a hypersurface of degree $2$, in which case
good bounds are known by \cite{ES}. Moreover, in the notation of Theorem \ref{newa-e}, if $k\ne e-2$, then
$t\geq 3$, and in this case we have essentially removed a factor of $(d!)^d$ from the denominator of the
estimate given in Theorem \ref{newa-e}.}
\end{remark}

\begin{remark}{\rm We can also improve the statements by bringing in the idea of the \it core \rm of the
maximal ideal, which is the intersection of all reductions of the maximal ideal. The $t$ we choose in the
statement of Theorem \ref{mainthm} can actually be chosen maximal so that $\overline{\m^t}$ is not
contained in the core of $\m$; then there will be some minimal reduction which does not contain $\overline{\m^t}$.
} \end{remark}

\begin{acknowledgement} The authors thank the referee for a careful reading of the paper and pointing out the statement
and proof of Corollary~\ref{ref}. The second and third authors gratefully acknowledge support of
the NSF during this research. 
\end{acknowledgement}

\end{document}